\documentclass[12pt]{article}
\usepackage{amsthm, amssymb}
\usepackage{amsmath}
\usepackage{amsfonts}
\usepackage{amsthm}
\usepackage{mathrsfs}
\usepackage{verbatim}
\usepackage{xspace}
\usepackage{graphicx}
\usepackage[small]{caption}
\usepackage{color}
\usepackage{float}
\usepackage{wrapfig}

\newcommand{\breakingcomma}{  \begingroup\lccode`~=`,
  \lowercase{\endgroup\expandafter\def\expandafter~\expandafter{~\penalty0 }}}
\newtheorem{theorem}{Theorem}[section]
\newtheorem{lemma}[theorem]{Lemma}
\newtheorem{remark}[theorem]{Remark}
\newtheorem{proposition}[theorem]{Proposition}
\newtheorem{corollary}[theorem]{Corollary}
\newtheorem{definition}[theorem]{Definition}
\newtheorem{example}[theorem]{Example}
\newcount\refno
\newcommand\be{\begin{equation}}
\newcommand\ee{\end{equation}}
\newcommand\bn{\begin{eqnarray}}
\newcommand\en{\end{eqnarray}}
\newcommand\bns{\begin{eqnarray*}}
\newcommand\ens{\end{eqnarray*}}
\newcommand\bd{\begin{definition}}
\newcommand\ed{\end{definition}}
\newcommand\br{\begin{remark}}
\newcommand\er{\end{remark}}
\newcommand\bt{\begin{theorem}}
\newcommand\et{\end{theorem}}
\newcommand\bp{\begin{proposition}}
\newcommand\ep{\end{proposition}}
\newcommand\bc{\begin{corollary}}
\newcommand\ec{\end{corollary}}
\newcommand\bl{\begin{lemma}}
\newcommand\el{\end{lemma}}

\newcommand{\op}{\mbox{one-$p$th}}

\makeindex
\begin{document}

\title{One-$p$th Riordan Arrays in the Construction of Identities}
%and their Applications }
\author{Tian-Xiao He\\
{\small Department of Mathematics}\\
{\small Illinois Wesleyan University, Bloomington, Illinois 61702, USA}\\
}
\date{In Memory of Professor Leetsch C. Hsu}
\maketitle

\begin{abstract}
\noindent 
For an integer $p\geq 2$ we construct vertical and horizontal {\op} Riordan arrays from a Riordan array. When $p=2$ {\op} Riordan arrays reduced to well known half Riordan arrays. The generating functions of the $A$-sequences of vertical and horizontal {\op} Riordan arrays are found. The vertical and horizontal {\op} Riordan arrays provide an approach to construct many identities. They can also be used to verify some well known identities readily. 


\end{abstract}

%%%%%%%%%%%%%%%%%%%%%%%%%%%%%%%%%%%%%%%%%%%%%%%%%%%%%%%%%%%%%%%%%%%%%%

%%%%%%%%%%%%%%%%%%%%%%%%%%%%%%%%%%%%%%%%%%%%%%%%%%%%%%%%%%%%%%%%%%%%%%
\setcounter{page}{1} \pagestyle{myheadings} 
\markboth{T. X. He}
{Half Riordan Arrays}

%%%%%%%%%%%%%%%%%%%%%%%%%%%%%%%%%%%%%%%%%%%%%%%%%%%%%%%%%%%%%%%%%%%%%%
%INSERT ABSTRACT HERE

%\begin{document}
\section{Introduction}

The Riordan group is a group of infinite lower triangular matrices defined
by two generating functions. Let $g\left(
z\right)=g_{0}+g_{1}z+g_{2}z^{2}+\cdots $ and $f\left(
z\right)=f_{1}z+f_{2}z^{2}+\cdots $ with $g_{0}$ and $f_{1}$ nonzero.
Without much loss of generality we will also set $g_{0}=1$. Given $g\left(
z\right) $ and $f\left( z\right),$ the matrix they define is $D=\left(
d_{n,k}\right)_{n,k\geq 0}$, where $d_{n,k}=\left[ z^{n}\right] g\left(
z\right) f\left( z\right) ^{k}$. For the sake of readability we often
shorten $g\left(z\right) $ and $f\left( z\right) $ to $g$ and $f$ and we
will denote $D$ as $\left( g,f\right)$. Essentially the columns of the
matrix can be thought of as a geometric sequence with $g$ as the leading term
and $f$ as the multiplier term. Two examples are the identity matrix

\begin{equation*}
\left( 1,z\right) =\left[ 
\begin{array}{ccccc}
1 & 0 & 0 & 0 &  \\ 
0 & 1 & 0 & 0 &  \\ 
0 & 0 & 1 & 0 & \cdots \\ 
0 & 0 & 0 & 1 &  \\ 
&  & \cdots &  & \ddots 
\end{array}
\right]
\end{equation*}
and the Pascal matrix

\begin{equation*}
\left( \frac{1}{1-z},\frac{z}{1-z} \right) =\left[ 
\begin{array}{ccccc}
1 & 0 & 0 & 0 & \cdots \\ 
1 & 1 & 0 & 0 &  \\ 
1 & 2 & 1 & 0 &  \\ 
1 & 3 & 3 & 1 &  \\ 
&  & \ldots &  & \ddots%
\end{array}
\right].
\end{equation*}

Here is a list of six important subgroups of the Riordan group (see \cite{SGWW}).

\begin{itemize}
\item the {\it Appell subgroup} $\{ (g(z),\,z)\}$.
\item the {\it Lagrange (associated) subgroup} $\{(1,\,f(z))\}$.
\item the {\it $k$-Bell subgroup} $\{(g(z),\, z(g(z))^k)\}$, where $k$ is a fixed positive integer. 
\item the {\it hitting-time subgroup} $\{(zf'(z)/f(z),\, f(z))\}$.
\item the {\it derivative subgroup} $\{ (f'(z), \, f(z))\}$.
\item the {\it checkerboard subgroup} $\{ (g(z),\, f(z))\},$ where $g$ is an even function and $f$ is an odd function. 
\end{itemize}

The $1$-Bell subgroup is referred to as the Bell subgroup for short, and the Appell subgroup can be considered as the $0$-Bell subgroup if we allow $k=0$ to be included in the definition of the $k$-Bell subgroup.

\bigskip The Riordan group acts on the set of column vectors by matrix
multiplication. In terms of generating functions we let $d\left(
z\right)=d_{0}+d_{1}z+d_{2}z^{2}+\cdots $ and $\ h\left(
z\right)=h_{0}+h_{1}z+h_{2}z^{2}+\cdots$. If $\left[ d_{0},d_{1},d_{2}, 
\cdots \right] ^{T}$ and $\left[ h_{0},h_{1},h_{2},\cdots \right] ^{T}$ are
the corresponding column vectors we observe that

\begin{equation*}
\left( g,f\right) \left[d_{0},d_{1},d_{2},\cdots \right] ^{T}=\left[
h_{0},h_{1},h_{2},\cdots \right] ^{T}
\end{equation*}
translates to

\begin{equation*}
d_{0}g\left( z\right) +d_{1}g\left( z\right) f\left( z\right) +d_{2}g\left(
z\right) f\left( z\right) ^{2}+\cdots =g\left( z\right) \cdot d\left(
f(z\right) )=h\left( z\right).
\end{equation*}
This simple observation is called the Fundamental theorem of Riordan Arrays
and is abbreviated as FTRA.

The first application of the fundamental theorem is to set $d\left( z\right)
=\hat g\left( z\right) \hat f\left( z\right) ^{k}$ so that

\begin{equation*}
h\left( z\right) =g\left( z\right) \cdot \hat g\left( f\left( z\right) \right)
\hat f\left( f\left( z\right) \right) ^{k}.
\end{equation*}
As $k$ ranges over $0,1,2,\cdots $ the multiplication rule for Riordan
arrays emerges.

Riordan arrays play an important unifying role in enumerative combinatorics, especially in proving combinatorial identities, for instance, some results presented in \cite{Spr95}, \cite{He18}, \cite{Hsu15}, etc. This paper will define a new type of Riordan arrays and study their applications in the construction of identities. 

We define the \textit{Riordan group} as the set of all pairs $(g,f)$ as
above together with the multiplication operation

\begin{equation*}
(g,f) (\hat g,\hat f)=(g\cdot (\hat g\circ f),\hat f\circ f).
\end{equation*}
The identity element for this group is $(1,z)$. If we denote the 
compositional inverse of $f$ as $\bar{f}$, then

\begin{equation*}
(g,f)^{-1}=\left( \frac{1}{g\circ \bar{f}}\,\text{ }\bar{f}\right).
\end{equation*}

As an example we return to the Pascal matrix where $f=\frac{z}{1-z}$. 
The inverse is $\overline{f}=\frac{z}{1+z},$ $g\left( \overline{f}\right) =
\frac{1}{1-\left( \frac{z}{1+z}\right) }=1+z$ and the inverse matrix starts 
\begin{equation*}
\left( \frac{1}{1+z},\frac{z}{1+z}\right) =\left[ 
\begin{array}{ccccc}
1 & 0 & 0 & 0 & 0 \\ 
-1 & 1 & 0 & 0 & 0 \\ 
1 & -2 & 1 & 0 & 0 \\ 
-1 & 3 & -3 & 1 & 0 \\ 
1 & -4 & 6 & -4 & 1
\end{array}
\right].
\end{equation*}
Both Pascal matrix and $(1/(1+z), z/(1+z))$ are pseudo-involution Riordan array 
due to their multiplications with $(1, -z)$ are involutions. 

For more information about the Riordan group see Shapiro, Getu, Woan and
Woodson \cite{SGWW}, Shapiro \cite{Shapiro}, Barry \cite{Barry}, and Zeleke \cite{Zeleke}. 
Shapiro and the author presented palindromes of pseudo-involutions in a recent paper \cite{HS20}. For general information about such items as Catalan numbers, Motzkin numbers,
generating functions and the like there are many excellent sources including
Stanley \cite{Stanley, Sta} and Aigner \cite{Aig}. A short survey and an extension of 
Catalan numbers and Catalan matrices can be seen in \cite{He13, HS17}. Fundamental papers by Sprugnoli \cite{Spr94, Spr95} investigated the Riordan arrays and showed that they constitute a practical device for solving {\it combinatorial sums}\index{combinatorial sums} by means of the generating functions and the {\it Lagrange inversion formula}. 

For a function $f$ as above, there is a sequence $a_{0},a_{1},a_{2},\cdots $
called the $A$ sequence such that

\begin{equation*}
f=z\left( a_{0}+a_{1}f+a_{2}f^{2}+a_{3}f^{3}+\cdots \right).
\end{equation*}
The corresponding generating function is $A\left( z\right)
=\sum\nolimits_{n\geq 0}a_{n}z^{n}$ so we have, in terms of generating
functions, $f=zA(f)$. See Merlini, Rogers, Sprugnoli, and Verri, \cite{MRSV}
for a proof and Sprugnoli and the author \cite{HS} and the author \cite{He20} for further results. The $A$
sequence enables us to inductively compute the next row of a Riordan matrix
since

\begin{equation*}
d_{n+1,k}=a_{0}d_{n,k-1}+a_{1}d_{n,k}+a_{2}d_{n,k+1}+\cdots.
\end{equation*}
The missing item is for the left most, i.e., zeroth column and there is a
second sequence, the $Z$ sequence such that

\begin{equation*}
d_{n+1,0}=z_{0}d_{n,0}+z_{1}d_{n,1}+z_{2}d_{n,2}+\cdots.
\end{equation*}
The generating function $Z=\sum_{n\geq 0}z_{n}z^{n}$ is defined by the
equation $g(z)=1/(1-zZ(f(z)))$.

By substituting $z=\overline{f}$ into the equation $f = z(A(f))$, we may have  $z=\overline{f}(z)A(z) =\overline{f}A$. Similarly, applying $\overline{f}$ gives us a useful alternate form of $g(z)=1/(1-z Z(f(z)))$ as $Z=(g(\bar f)-1)/(\bar f g(\bar f))$. We call $A(z)$ and $Z(z)$ the $A$ and $Z$ functions of the Riordan array $(g,f)$. 

We now consider an extension of Riordan arrays called half Riordan arrays, which will be extended to one-$p$th Riordan arrays in the next section. 

The entries of a Riordan array have a multitude of interesting combinatorial explanations. The central entries play a significant role. For instance, the central entries of the Pascal matrix $(1/(1-z), z/(1-z))$ are the central binomial coefficients $\binom{2n}{n}$ (see the sequence A000984 in OEIS \cite{OEIS}) that can be explained as the number of ordered trees with a distinguished point. In addition, its exponential generating function is a modified Bessel function of the first kind. Similarly, the central entries of the Delannoy matrix $(1/(1-z) , z(1+z)/(1-z))$, called the Pascal-like Riordan array, are the central Delannoy numbers $\sum^n_{k=0}\binom{n}{k}^22^k$ (see the sequence A001850 in OEIS \cite{OEIS}). The central Delannoy numbers can be explained as the number of paths from $(0,0)$ to $(n,n)$ in an $n\times n$ grid using only steps north, northeast and east (i.e., steps $(1,0)$, $(1,1)$, and $(0,1)$). In addition, the $n$th central Delannoy numbers is the $n$th Legendre polynomial's value at $3$. It is interesting, therefore, to be able to give generating functions of such central terms in a systematic way. In recent papers \cite{Barry13, Barry19, Bar, He20-2, 
%YDYY, YXG, 
YXH, YZYH} (cf. also the references of \cite{He20-2}), it has been shown how to find generating functions of the central entries of some Riordan arrays.

Yang, Zheng, Yuan, and the author \cite{YZYH} give the following definition of half Riordan arrays (HRAs), which are called vertical half Riordan arrays in Barry \cite{Barry19} and in  \cite{He20}. 

\begin{definition} \label{def:3.1}
Let $(g,f)=(d_{n,k})_{n,k\geq 0}$ be a Riordan array. Its related half Riordan array $(v_{n,k})_{n,k\geq 0}$, 
called the vertical half Riordan array (VHRA), is defined by 

\be\label{3.1}
v_{n,k}=d_{2n-k,n}.
\ee
\end{definition}

Denote $\phi=\overline{t^2/f}$. A direct approach is used in \cite{He20-2} to show that $(v_{n,k})_{n,k\geq 0}=(t\phi'(t)g(\phi)/\phi, \phi)$ based on the Lagrange inversion formula.

In \cite{He20-2}, a decomposition of $(v_{n,k})_{n,k\geq 0}$  is presented as 

\be\label{3.4}
\left( \frac{t\phi'(t)g(\phi)}{\phi}, \phi\right)=\left( \frac{t\phi'}{\phi},\phi\right)(g,t).
\ee
Decomposition \eqref{3.4} suggests a more general type of half of Riordan array $(g,f)$ defined by 

\be\label{3.4-2}
\left( \frac{t\phi'(t)g(\phi)}{\phi}, f(\phi)\right)=\left( \frac{t\phi'}{\phi},\phi\right)(g,f),
\ee
which is called the horizontal half of Riordan array (HHRA) in \cite{Barry19, He20}, in order to distinguish it from VHRA. A similar approach can be used to show that the entries of the HHRA $(h_{n,k})_{n,k\geq 0}$ of $(g,f)=(d_{n,k})_{n,k\geq 0}$ are 

\be\label{3.4-3}
h_{n,k}=d_{2n, n+k},
\ee
while a constructive approach is presented in \cite{Barry19} and an $(m,r)$ extension can be seen in \cite{YXH}. 

In the next section the VHRA and the HHRA of a given Riordan array will be extended to the {\op} vertical and the {\op} horizontal Riordan arrays of the Riordan array. Then the {\op} vertical transformation operators and the {\op} horizontal Riordan array transformation operators will be defined. We will present the relationship between the two types of {\op} Riordan arrays by using their matrix factorization and the Lagrange inversion formula.  In Section $3$, the sequence characterizations of the two types of {\op} Riordan arrays and several illustrating examples are given. In Section $4$, we study transformations among Riordan arrays by using the {\op} Riordan array operators. The conditions for transforming Riordan arrays to pseudo-involution Riordan arrays by using the {\op} Riordan arrays are given.  The condition for preserving the elements of a certain subgroup of the Riordan group under the {\op}f Riordan array transformation is shown. Other properties of the halves of Riordan arrays and their entries such as related recurrence relations, double variable generating functions, combinatorial explanations are also studied in the section. In the last section, we will show the construction of identities and summation formulae by using {\op} Riordan arrays. 

\section{One-$p$th Riordan arrays}

The vertical and horizontal one-$p$th Riordan arrays of a Riordan array $(g,f)$ will be defined and constructed in the following two theorems.

\begin{theorem}\label{thm:4.3.6-2}
Given a Riordan array $(d_{n,k})_{n,k\geq 0}=(g,f)$, for any integers $p\geq 1$ and $r\geq 0$,  $(\widehat {d}_{n,k}=d_{pn+r-k,(p-1)n+r})_{n,k\geq 0}$ defines a new Riordan array, called the one-$p$th or $(p,r)$ vertical Riordan array of $(g,f)$, which can be written as 

\be\label{p-1}
\left( \frac{t\phi'(t)g(\phi)f(\phi)^r}{\phi^{r+1}}, \phi\right),\quad \mbox{where} \quad \phi(t)=\overline{\frac{t^{p}}{f(t)^{p-1}}},
\ee
and $\bar h(t)$ is the compositional inverse of $h(t)$ $(h(0)=0$ and $h'(0)\not=0)$. Particularly, if $p=1$ and $r=0$, then $(\widehat{d}_{n,k}=d_{n-k,0})_{n,k\geq 0}$ is the Toeplitz matrix (or diagonal-constant matrix) of the $0$th column of $(d_{n,k})_{n,k\geq 0}$, and if $p=2$ and $r=0$, then $(\widehat{d}_{n,k}=d_{2n-k,n})_{n,k\geq 0}$ is the VHRA of the Riordan array $(d_{n,k})_{n,k\geq 0}$.

Moreover, the generating function of the A-sequence of the new array is $(A(f))^{p-1}=(f/t)^{p-1}$, where $A(t)$ is the generating function of the A-sequence of the given Riordan array.
\end{theorem}

The Lagrange Inverse Formula (LIF) will be used in the proof. Let $F(t)$ be any formal power series, and let $\phi(t)$ and $u(t)=f(t)/t$ satisfy $\phi=tu(\phi)$. Then the following LIF holds (see, for example, $K6'$ in Merlini, Sprugnoli, and Verri \cite{MSV}).

\be\label{3.2}
[t^n]F(\phi(t))=[t^n]F(t)u(t)^{n-1}(u(t)-tu'(t)).
\ee

\begin{proof}
From $\phi(t)=\overline{t^{p}/f(t)^{p-1}}$ we have $\bar \phi(t)=t^p/f(t)^{p-1}$ and consequently 
$t=\phi(t)^p/f(\phi(t))^{p-1}$. Hence, we may write  

\[
\phi=t u(\phi)\quad \mbox{where}\quad u(t)=\left(\frac{f(t)}{t}\right)^{p-1}.
\]
Taking derivative on the both sides of the equation $\phi=t u(\phi)$ and noting the definition of $u(t)$, we obtain 

\begin{align*}
\phi'(t)=&\left(\frac{f(\phi)}{\phi}\right)^{p-1}+t(p-1)\left(\frac{f(\phi)}{\phi}\right)^{p-2}\frac{f'(\phi)\phi'(t)\phi-\phi'(t)f(\phi)}{\phi^2},
\end{align*}
which yields 

\[
\phi'(t)=\left. \left(\frac{f(\phi)}{\phi}\right)^{p-1}\right/\left( 1-t(p-1)\left(\frac{f(\phi)}{\phi}\right)^{p-2}\frac{f'(\phi)\phi-f(\phi)}{\phi^2}\right).
\]
Noting $t=\phi/u(\phi)=\phi^p/f(\phi)^{p-1}$, the last expression devotes  

\begin{align}\label{p-2}
\phi'(t)=&\left. \left(\frac{f(\phi)}{\phi}\right)^{p-1}\right/\left( 1-\frac{p-1}{f(\phi)}(f'(\phi)\phi-f(\phi))\right)\nonumber\\
=&\frac{(f(\phi))^p}{\phi^{p-1}( f(\phi) -(p-1)(\phi f'(\phi)-f(\phi)))}
\end{align}
We now use \eqref{p-2}, $t=\phi^p/f(\phi)^{p-1}$, and the LIF shown in \eqref{3.2} to calculate $\widehat d_{n,k}$ for $n,k\geq 0$ 

\begin{align*}
\widehat {d}_{n,k}=&[t^n] \frac{t\phi'(t)g(\phi)f(\phi)^r}{\phi^{r+1}}\left( \phi\right)^k\\
=&[t^n] \frac{\phi^p}{(f(\phi))^{p-1}}\frac{\phi^k(f(\phi))^{p+r}g(\phi)}{\phi^{p+r}(f(\phi) -(p-1)(\phi f'(\phi)-f(\phi)))}\\
=&[t^n]\frac{(f(\phi))^{r+1}g(\phi)}{\phi^{r-k}(f(\phi) -(p-1)(\phi f'(\phi)-f(\phi)))}\\
=&[t^n]\frac{(f(t))^{r+1}g(t)}{t^{r-k}(f(t) -(p-1)(t f'(t)-f(t)))}u(t)^{n-1}(u(t)-tu'(t)),
\end{align*}
where $u(t)=\left(\frac{f(t)}{t}\right)^{p-1}$ and 

\[
u'(t)=(p-1)\left(\frac{f(t)}{t}\right)^{p-2}\frac{tf'(t)-f(t)}{t^2}.
\]
Substituting the expressions of $u(t)$ and $u'(t)$ into the rightmost expression of $\widehat d_{n,k}$, we have

\begin{align*}
\widehat {d}_{n,k}=&[t^n]\frac{(f(t))^{r+1}g(t)}{t^{r-k}(f(t) -(p-1)(t f'(t)-f(t)))}\frac{(f(t))^{(p-1)(n-1)}}{t^{(p-1)(n-1)}}\\
&\quad \times \left( \frac{(f(t)^{p-1}}{t^{p-1}}-t(p-1)\frac{(f(t))^{p-2}}{t^{p-2}}\frac{tf'(t)-f(t)}{t^2}\right)\\
=&[t^n] \frac{(f(t)^{(p-1)(n-1)+r+1}g(t)}{t^{(p-1)(n-1)+r-k}(f(t)-(p-1)(tf'(t)-f(t)))}\\
&\quad \times \frac{(f(t))^{p-2}}{t^{p-1}}\left(f(t)-(p-1)(tf'(t)-f(t))\right)\\
=&[t^n]g(t)\frac{(f(t))^{(p-1)n+r}}{t^{(p-1)n+r-k}}=[t^{pn+r-k}]g(t)(f(t))^{(p-1)n+r}=d_{pn+r-k,(p-1)n+r}.
\end{align*}

Particularly, if $p=1$ and $r=0$, then $(\widehat {d}_{n,k}=d_{n-k,0})_{n,k}$ is the Toeplitz matrix of the $0$th column of $(g,f)$. If $p=2$ and $r=0$, then $(\widehat {d}_{n,k}=d_{2n-k, n})_{n,k\geq 0}$ is the VHRA of $(g,f)$.

As for the $\widehat{A}_p$, the generating function of the $A$-sequence of $(\widehat {d}_{n,k})_{n,k\geq 0}$, we have $t\widehat {A}_p(\phi)=\phi$, which implies $\widehat{A}_p(t)=t/(t^p/f^{p-1})$, or equivalently,

\[
\widehat{A}_p(\bar f)=\left(\frac{t}{\bar f}\right)^{p-1}=(A(t))^{p-1}.
\]
Hence, $\widehat{A}_p(t)=(A(f))^{p-1}=(f/t)^{p-1}$ because $tA(f)=f$, completing the proof of the theorem.
\end{proof}

\begin{theorem}\label{thm:4.3.6-3}
Given a Riordan array $(d_{n,k})_{n,k\geq 0}=(g,f)$, for any integers $p\geq 1$ and $r\geq 0$,  $(\tilde d_{n,k}=d_{pn+r,(p-1)n+r+k})_{n,k\geq 0}$ defines a new Riordan array, called the one-$p$th or $(p,r)$ horizontal Riordan array of $(g,f)$, which can be written as 

\be\label{p-1-2}
\left( \frac{t\phi'(t)g(\phi)f(\phi)^r}{\phi^{r+1}}, f(\phi)\right),\quad \mbox{where} \quad \phi(t)=\overline{\frac{t^{p}}{f(t)^{p-1}}},
\ee
and $\bar h(t)$ is the compositional inverse of $h(t)$ $(h(0)=0$ and $h'(0)\not=0)$. Particularly, if $p=1$ and $r=0$, the {\op} Riordan array reduces to the given Riordan array, and if $p=2$ and $r=0$, the {\op} Riordan array is the HHRA of the given Riordan array.

Moreover, the generating function of the A-sequence of the new array is $(A(t))^p$ , where $A(t)$ is the generating function of the A-sequence of the given Riordan array.
\end{theorem}

\begin{proof}
We now use \eqref{p-2} above, $t=\phi^p/f(\phi)^{p-1}$, and the LIF shown in \eqref{3.2} to calculate $\tilde d_{n,k}$ for $n,k\geq 0$ 

\begin{align*}
\tilde d_{n,k}=&[t^n] \frac{t\phi'(t)g(\phi)f(\phi)^r}{\phi^{r+1}}\left( f(\phi)\right)^k\\
=&[t^n] \frac{\phi^p}{(f(\phi))^{p-1}}\frac{(f(\phi))^{p+r+k}g(\phi)}{\phi^{p+r}(f(\phi) -(p-1)(\phi f'(\phi)-f(\phi)))}\\
=&[t^n]\frac{(f(\phi))^{r+k+1}g(\phi)}{\phi^{r}(f(\phi) -(p-1)(\phi f'(\phi)-f(\phi)))}\\
=&[t^n]\frac{(f(t))^{r+k+1}g(t)}{t^{r}(f(t) -(p-1)(t f'(t)-f(t)))}u(t)^{n-1}(u(t)-tu'(t)),
\end{align*}
where $u(t)=\left(\frac{f(t)}{t}\right)^{p-1}$ and from the proof of Theorem \ref{thm:4.3.6-2} 

\[
u'(t)=(p-1)\left(\frac{f(t)}{t}\right)^{p-2}\frac{tf'(t)-f(t)}{t^2}.
\]
Substituting the expressions of $u(t)$ and $u'(t)$ into the rightmost expression of $\tilde d_{n,k}$, we have

\begin{align*}
\tilde d_{n,k}=&[t^n]\frac{(f(t))^{r+k+1}g(t)}{t^{r}(f(t) -(p-1)(t f'(t)-f(t)))}\frac{(f(t))^{(p-1)(n-1)}}{t^{(p-1)(n-1)}}\\
&\quad \times \left( \frac{(f(t)^{p-1}}{t^{p-1}}-t(p-1)\frac{(f(t))^{p-2}}{t^{p-2}}\frac{tf'(t)-f(t)}{t^2}\right)\\
=&[t^n] \frac{(f(t)^{(p-1)(n-1)+r+k+1}g(t)}{t^{(p-1)(n-1)+r}(f(t)-(p-1)(tf'(t)-f(t)))}\\
&\quad \times \frac{(f(t))^{p-2}}{t^{p-1}}\left(f(t)-(p-1)(tf'(t)-f(t))\right)\\
=&[t^n]g(t)\frac{(f(t))^{(p-1)n+r+k}}{t^{(p-1)n+r}}=[t^{pn+r}]g(t)(f(t))^{(p-1)n+r+k}=d_{pn+r,(p-1)n+r+k}.
\end{align*}
Particularly, if $p=1$ and $r=0$, then $\tilde d_{n,k}=d_{n,k}$, while $p=2$ and $r=0$ yields $\tilde d_{n,k}=d_{2n, n+k}$, the $(n,k)$ entry of the HHRA of $(g,f)$.

Let $A(t)$ be the generating function of the $A$-sequence of the given Riordan array $(g,f)$. Then $A(f(t))=f(t)/t$. Let $A_p(t)$ be the generating function of the $A$-sequence of the Riordan array 
shown in \eqref{p-1}. Then $A_p(f(\phi))=\frac{f(\phi)}{t}$. Substituting $t=\overline{\phi}(t)$ into the last equation yields

\[
A_p(f)=\frac{f(t)}{\overline{\phi}(t)}=\frac{f(t)}{t^p/(f(t))^{p-1}}=\left( \frac{f(t)}{t}\right)^p=(A(f))^p,
\]
i.e., $A_p(t)=(A(t))^p$ completing the proof.
\end{proof}

\section{Identities related to {\op} Riordan arrays}

We may use Theorems \ref{thm:4.3.6-2} and \ref{thm:4.3.6-3} and the Fa\`a di Bruno formula to establish a class of summation formulae. 

Let $h(t)=\sum^\infty_{n=0}\alpha_nt^n$ be a given formal power series with the case $h(0)=\alpha_0\not= 0$. Assume that $f(a+t)$ has a formal power series expansion in $t$ with $a\in{\mathbb R}$, real numbers, and let $\bar f$ denote the compositional inverse of $f$ so that $(\bar f\circ f)(t)=(f\circ \bar f)(t)=t$. Then the composition of $f$ and $h$ in the case of $h(0)=a$ 
still possess a formal series expansion in $t$, namely, 

\begin{align}\label{p-3}
(f\circ h)(t)=&\sum^\infty_{n=0}\left([t^n](f\circ h)(t) \right)t^n=f\left( a+\sum^\infty_{n=1}\alpha_n t^n\right)\nonumber\\
=&f(a)+\sum^\infty_{n=1}\left( [t^n] (f\circ h)(t)\right)t^n.
\end{align}

Let $f^{(k)}(a)$ denote the $k$th derivative of $f(t)$ at $t=a$, i.e.,  

\[
f^{(k)}(a)=(d^k/dt^k)f(t)|_{t=a}.
\]
Recall the Fa\`a di Brumo's formula when applied to $(f\circ h)(t)$ may be written in the form (cf. Section 3.4 of \cite{Com74}) 

\be\label{p-4}
[t^n](f\circ \phi)=\sum_{\sigma(n)}f^{(k)}(\phi(0))\Pi^n_{j=1}\frac{1}{k_j!}\left([t^i]\phi\right)^{k_j},
\ee
where the summation ranges over the set $\sigma (n)$ of all partitions of $n$, that is, over the set of all nonnegative integral solutions $(k_1,k_2,\ldots,k_n)$ of the equations $k_1+2k_2+\cdots+ nk_n=n$ and $k_1+k_2+\cdots+k_n=k$, $k=1,2,\ldots, n$. Each solution $(k_1,k_2,\ldots, k_n)$ of the equations is called a partition of $n$ with $k$ parts and is denoted by $\sigma(n,k)$. Of course, the set $\sigma(n)$ is the union of all subsets $\sigma(n,k)$, $k=1,2,\ldots, n$.

Let $\beta_n=[t^n](f\circ h)(t)$ and $h(0)=\alpha_0=a$. Then there exists a pair of reciprocal relations

\begin{align}
&\beta_n=\sum_{\sigma(n)}f^{(k)}(a)\frac{\alpha^{k_1}_1\cdots \alpha^{k_n}_n}{k_1!\cdots k_n!},\label{Ex:7-3-1}\\
&\alpha_n=\sum_{\sigma(n)}\bar f^{(k)}(f(a))\frac{\beta^{k_1}_1\cdots \beta^{k_n}_n}{k_1!\cdots k_n!},\label{Ex:7-3-2}
\end{align}
where the summation ranges the set $\sigma (n)$ of all partitions of $n$. In fact, from \eqref{p-3} the given conditions ensure that there holds a pair of formal series expansions

\begin{align}
&f\left( a+\sum_{n\geq 1}\alpha_nt^n\right)=f(a)+\sum_{n\geq 1}\beta_nt^n,\label{Ex:7-3-5}\\
&\bar f\left( f(a)+\sum_{n\geq 1}\beta_nt^n\right)=a+\sum_{n\geq 1}\alpha_nt^n.\label{Ex:7-3-6}
\end{align}
Thus, an application of the Fa\`a di Bruno formula \eqref{p-4} to $(f\circ \phi)(t)$, on the LHS of \eqref{Ex:7-3-5} yields the expression \eqref{Ex:7-3-1} with $[t^i]\phi=\alpha_i$, 
$[t^n](f\circ \phi)=\beta_n$, and $\phi(0)=a$. Note that the LHS of \eqref{Ex:7-3-6} may be expressed as $\phi(t)=((\bar f\circ f)\circ \phi)(t)=(\bar f\circ (f\circ \phi))(t)$, so that in a like manner and application of the Fa\`a di Bruno formula to the LHS of \eqref{Ex:7-3-6} gives precisely the equality \eqref{Ex:7-3-2}.

Replacing $\alpha_n$ by $x_n/n!$ and $\beta_n$ by $y_n/n!$, we see that \eqref{Ex:7-3-1} and \eqref{Ex:7-3-2} may be expressed in terms of the exponential Bell polynomials, namely,

\begin{align}
&y_n=\sum^n_{k=1}f^{(k)}(a)B_{n,k}(x_1,x_2,\ldots,x_{n-k+1}),\label{Ex:7-3-3}\\
&x_n=\sum^n_{k=1}\bar f^{(k)}(a)B_{n,k}(y_1,y_2,\ldots,y_{n-k+1}),\label{Ex:7-3-4}
\end{align}
where $B_{n,k}(\ldots)$ is defined by (cf. Section 3.3 of \cite{Com74}) 

\[
B_{n,k}(x_1,x_2,\ldots,x_{n-k+1})=\sum_{\sigma(n,k)}\frac{n!}{k_1!k_2!\cdots}\left( \frac{x_1}{1!}\right)^{k_1}\left( \frac{x_2}{2!}\right)^{k_2}\cdots
\]
and $\sigma(n,k)$ as shown above is the set of the solutions of the partition equations for a given $k$ ($1\leq k\leq n$). $B_{n,k}=B_{n,k}(f_1,f_2,\ldots)$ is the Bell polynomial with respect to $(n!)_{n\in\mathbb{N}}$, defined as follows:

\begin{equation}\label{Bellptoom}
\frac{1}{k!}(f(z))^k=\sum_{n=k}^\infty B_{n,k}\frac{z^n}{n!}\,.
\end{equation}
Therefore, $B_{n,k}=[z^n/n!](f(z))^k/k!$, which implies that the iteration matrix $B(f(z))$ is the Riordan array $(1, f(z))$. Now, the following important property of the iteration matrix
(see Theorem A on p. 145 of Comtet \cite{Com74}, Roman \cite{Rom}, and Roman and Rota \cite{RomRot78} )

\[
B(f(g(z)))=B(g(z))B(f(z))
\]
is trivial in the context of the theory of Riordan arrays, i.e.,

\[
(1,f(g(z)))=(1,g(z))(1,f(z))\,;
\]
and the Fa\`{a} di Bruno formula derived from the above property is an application of the FTRA.

Let $f(x)=x^p$ ($p\not= 0$). Then $\bar f(x)=x^{1/p}$ with $f^{(k)}(1)=(p)_k$ and $\bar f^{(k)}(1)=(1/p)_k$. Hence, we obtain the special cases of \eqref{Ex:7-3-1} and \eqref{Ex:7-3-2}:

\begin{align}
&\beta_n=\sum_{\sigma(n)}(\alpha)_k\frac{\alpha_1^{k_1}\cdots \alpha_n^{k_n}}{k_1!\cdots k_n!},\label{Ex:7-3-1-2}\\
&\alpha_n=\sum_{\sigma(n)}(1/\alpha)_k\frac{\beta_1^{k_1}\cdots \beta_n^{k_n}}{k_1!\cdots k_n!},
\label{Ex:7-3-2-2}
\end{align}
where $(p)_k=p (p-1)\ldots (p-k+1)$ and $(\alpha)_0=1$. The above Fa\`a di Bruno's relations have the associated expressions 

\begin{align}
&\left( 1+\sum^\infty_{n=1}\alpha_nt^n\right)^p=1+\sum^\infty_{n=1}\beta_nt^n,\label{Ex:7-3-1-3}\\
&\left( 1+\sum^\infty_{n=1}\beta_nt^n\right)^{1/p}=1+\sum^\infty_{n=1}\alpha_nt^n.\label{Ex:7-3-2-3}
\end{align}

As an example, if $h=a_0+a_1t$ and $f(t)=t^p$, then $f(h(t))=a_0^p(1+\alpha_1 t)^p$, where 
$\alpha_1=a_1/a_0$. From \eqref{Ex:7-3-1-3} we have 

\[
(a_0+a_1t)^p=a_0^p\left(1+\alpha_1t\right)^p=a_0^p\left( 1+ \sum^\infty_{j=1}\beta_j t^j\right),
\]
where 

\[
\beta_j=\sum_{\sigma(j)}(\alpha)_k\frac{\alpha_1^{k_1}\cdots \alpha_n^{k_n}}{k_1!\cdots k_n!}=(p)_j\frac{\alpha_1^j}{j!}=\binom{p}{j}\alpha_1^j,
\]
which presents the obvious expression $(a_0+a_1 t)^p=a_0^p+\sum^p_{j=1}\binom{p}{j}a_0^{p-j}a_1^j t^j$.

Similarly, if $h=a_0+a_1t+a_2t^2$, $a_0\not=0$, then 

\[
(a_0+a_1t+a_2t^2)^p=a_0^p\left(1+\frac{a_1}{a_0}t+\frac{a_2}{a_0}t^2\right)^p=
a_0^p\left( 1+\sum^p_{j=1}\beta_jt^j\right),
\]
where 

\[
\beta_j=\sum_{\sigma(j)}(p)_j\frac{1}{j_i!j_2!}\left(\frac{a_1}{a_0}\right)^{j_1}\left(\frac{a_2}{a_0}\right)^{j_2}=\sum^j_{j_i=0}\binom{p}{j}\binom{j}{j_1}\left(\frac{a_1}{a_0}\right)^{j_1}\left(\frac{a_2}{a_0}\right)^{j-j_1}.
\]

\begin{theorem}\label{thm:4.3.6-4}
Let $A(t)=\sum_{n\geq 0} a_nt^n$ $(a_0\not= 0)$ be the generating function of the $A$-sequence of the given Riordan array $(d_{n,k})_{n,k\geq 0}=(g,f)$, and let $(\tilde d_{n,k}=d_{pn+r,(p-1)n+r+k})_{n,k\geq 0}$ be the $(p,r)$ Riordan array of $(g,f)$. Then there exists the following summation formula:

\be\label{p-5}
d_{p(n+1)+r,(p-1)(n+1)+r+k+1}=\sum^{n-k}_{j= 0}\beta_jd_{pn+r,(p-1)n+r+k+j},
\ee
where by denoting $(p)_j=p(p-1)\ldots (p-j+1)$, $\beta_0=a_0^p$, and for $n\geq 1$ and 
$\alpha_i=a_i/a_0$, 

\begin{align}\label{p-6}
\beta_j=&a_0^p[t^j](A(t))^p=\sum_{\sigma(j)}(p)_j\frac{\alpha_1^{k_1}\cdots \alpha_j^{k_j}}{k_1!\cdots k_j!}\nonumber\\
=&\sum^j_{i=1}\sum_{\sigma(j,i)}\binom{p}{j}\frac{j!}{k_1!k_2!\ldots}(\alpha_1)^{k_1}(\alpha_2)^{k_2}\ldots.
\end{align}
Particularly, for $A(t)=a_0+a_1t$ and $A(t)=a_0+a_1t+a_2t^2$, we have 

\begin{align*}
&\beta_j=\binom{p}{j}a_0^{p-j}a_1^j\quad \mbox{and}\\
&\beta_j=\sum^j_{i=0}\binom{p}{j}\binom{j}{i}a_0^{p-j}a_1^{j-i}a_2^{i},
\end{align*}
respectively.
\end{theorem}

\begin{proof}
Since $(f(t))^p$ is the generating function of the $A$-sequence of $(\tilde d_{n,k})$ and $\tilde d_{n,k}=d_{pn+r,(p-1)n+r+k}$, we obtain \eqref{p-5} from the definition of $A$-sequence, where $\beta_j$ can be found from \eqref{p-3} and \eqref{Ex:7-3-1-2}. 
\end{proof}

Using \eqref{p-5} in Theorem \ref{thm:4.3.6-4}, one may obtain many identities. 

\begin{example}\label{ex:4.3.3}
Consider Pascal matrix $(1/(1-t), t/(1-t))$, its $A$-sequence generating function is $A(t)=1+t$. Applying \eqref{p-5}, we have 

\be\label{p-7}
\binom{p(n+1)+r}{(p-1)(n+1)+r+k+1}=\sum^{\min\{p,n-k\}}_{j=0}\binom{p}{j}\binom{pn+r}{(p-1)n+r+k+j}.
\ee 
If $p=1$ and $r=0$, the above identity reduces to the well-known identity $\binom{n+1}{k+1}=\binom{n}{k}+\binom{n}{k+1}$. 

The Riordan array $(1/(1-t-t^2), tC(t))$ is considered, where $C(t)=\sum^\infty_{n=0} \binom{2n}{n}t^n/(n+1)=(1-\sqrt{1-4t})/(2t)$ is the Catalan function. It can be found that the $A$-sequence of the Riordan array $(1/(1-t-t^2), tC(t))$ is $(1,1,1,\ldots)$, i.e., the $A$-sequence has the generating function $A(t)=1/(1-t)$. From \cite{GKP, HS17} we have 

\be\label{4.3.4}
C(t)^{k}=\sum_{n=0}^{\infty }\frac{k}{2n+k}\binom{2n+k}{n}t^{n}.  
\ee
Thus, the $(n,k)$ entry of the Riordan array $(1/(1-t-t^2), tC(t))$ is 

\begin{align*}
d_{n,k}=&[t^n] \frac{1}{1-t-t^2}(tC(t))^k\\
=&[t^{n-k}]\left(\sum_{i\geq 0}F_i t^i\right) 
\left( \sum_{j\geq 0}\frac{k}{2j+k}\binom{2j+k}{j}t^j\right)\\
=&[t^{n-k}]\sum_{i\geq 0} \left( \sum^i_{j=0}F_{i-j}\frac{k}{2j+k}\binom{2j+k}{j}\right)t^i\\
=&\sum^{n-k}_{j=0}F_{n-k-j}\frac{k}{2j+k}\binom{2j+k}{j}.\\
\end{align*}
Since 

\[
(A(t))^p=(1-t)^{-p}=\sum_{i\geq 0} \binom{-p}{i}(-t)^i=\sum_{i\geq 0} \binom{p+i-1}{i}t^i,
\]
From \eqref{p-5} there holds an identity 

\begin{align*}
&\sum^{n-k}_{j=0}F_{n-k-j}\frac{(p-1)(n+1)+r+k+1}{2j+(p-1)(n+1)+r+k+1}\binom{2j+(p-1)(n+1)+r+k+1}{j}\\
=&\sum_{i\geq0}\binom{p+i-1}{i}\sum^{n-k-i}_{j=0}F_{n-k-i-j}\frac{(p-1)n+r+k+i}{2j+(p-1)n+r+k+i}\binom{2j+(p-1)n+r+k+i}{j}.
\end{align*}

Similarly, for the Riordan array $(C(t), tC(t))$, its $(n,k)$ entry is 

\begin{align*}
d_{n,k}=&[t^n] t^k(C(t))^{k+1}\\
=&[t^{n-k}]\sum_{j\geq 0}\frac{k+1}{2j+k+1}\binom{2j+k+1}{j}t^j\\
=&\frac{k+1}{2n-k+1}\binom{2n-k+1}{n-k}.
\end{align*}
Hence, from \eqref{p-5} we may derive the identity 

\begin{align*}
&\frac{(p-1)(n+1)+r+k+2}{(p+1)(n+1)+r-k}\binom{(p+1)(n+1)+r-k}{n-k}\\
=&\sum^{n-k}_{j=0}
\frac{(p-1)n+r+k+j+1}{(p+1)n+r-k-j+1}\binom{p+j-1}{j}\binom{(p+1)n+r-k-j+1}{n-k-j}.
\end{align*}
\end{example}

\section{More identities}
The generating function $F_{m}(t)$ of the $m$th order Fuss-Catalan numbers $(F_{m}(n,$ $1))_{n\geq 0}$ is called the generalized binomial series in \cite{GKP}, and it satisfies the function equation $F_{m}(t)=1+tF_{m}(t)^{m}$. Hence from Lambert's formula for the Taylor expansion of the powers of $F_{m}(t)$ (cf. P. 201 of \cite{GKP}), we have 
\begin{equation}
F_{m}^{r}:=F_{m}(t)^{r}=\sum_{n\geq 0}\frac{r}{mn+r}\binom{mn+r}{n}t^{n}
\label{4.3.5}
\end{equation}
for all $r\in {{\mathbb{R}}}$, where $F_m(t)$ is defined by 
\be\label{4.3.5-2}
F_m(t)=\sum_{k\geq 0}\frac{(mk)!}{(m-1)k+1)!}\frac{t^k}{k!}=\sum_{k\geq 0} \frac{1}{(m-1)k+1}\binom{mk}{k}t^k.
\ee
For instance, 
\begin{align*}
&F_{0}(t)=1+t,\\
&F_1(t)=\sum_{k\geq 0} t^k=\frac{1}{1-t},\\
&F_2(t)=\sum_{k\geq 0} \frac{1}{k+1}\binom{2k}{k}t^k=C(t).
\end{align*}
The key case \eqref{4.3.5} leads the
following formula for $F_{m}(t)$:
\begin{equation}
F_{m}(t)=1+tF_{m}^{m}(t).  \label{4.3.6}
\end{equation}
Actually,
\begin{eqnarray*}
1+tF_{m}^{m}(t) &=&1+\sum_{n\geq 0}\frac{m}{mn+m}\binom{mn+m}{n}t^{n+1} \\
&=&1+\sum_{n\geq 1}\frac{m}{mn}\binom{mn}{n-1}t^{n} \\
&=&\sum_{n\geq 0}\frac{1}{mn+1}\binom{mn+1}{n}t^{n}=F_{m}(t).
\end{eqnarray*}
For the cases $m=1$ and $2$, we have $F_{1}=1/(1-t)$ and $F_{2}=C(t)$,
respectively. When $m=3$, the Fuss-Catalan numbers $\left( F_{3}\right) _{n}$
form the sequence $A001764$ (cf.  \cite{OEIS}), $1,1,3,12,55,273,1428,\ldots $, which are the {\it ternary numbers}\index{ternary numbers}. The ternary numbers count the number of $3$-Dyck paths or ternary paths. The generating function of the ternary numbers is denoted as $T(t)=\sum_{n=0}^{\infty }T_{n}t^{n}$ with $T_{n}=\frac{1}{3n+1}\binom{3n+1}{n}$, and is given equivalently by the equation $T(t)=1+tT(t)^3$. 

We now give more examples of Theorem \ref{thm:4.3.6-2} related to Fuss-Catalan numbers. First, we establish the relation between the Fuss-Catalan numbers and the Riordan array $(\tilde g, \tilde f)=(\tilde d_{n,k})_{n,k\geq 0}$, where $\tilde d_{n,k}=d_{pn+r,(p-1)n+r+k}$ and $d_{n,k}$ is the $(n,k)$ entry of the Pascal' triangle $(g,f)=(1/(1-t), t/(1-t))$. 

\begin{theorem}\label{thm:4.3.15}
Let $(d_{n,k})_{n,k\geq 0}=(1/(1-t), t/(1-t))$ be the Pascal triangle, for any integers $p\geq 2$ and $r\geq 0$ and a given Riordan array $(g,f)$ let $(\tilde d_{n,k}=d_{pn+r,(p-1)n+r+k})_{n,k\geq 0}=(\tilde g, \tilde f)$ be the one-$p$th or $(p,r)$ Riordan array of $(g,f)$.  Then 

\begin{align}
& \tilde g(t)=\sum_{n\geq 0}\binom{pn+r}{n}t^n=\left.\frac{(1+w)^{r+1}}{1-(p-1)w}\right|_{w=t(1+w)^p}\label{4.3.48}\\
&\tilde f(t)= \sum^\infty_{n=1}\frac{1}{pn+1}\binom{pn+1}{n}t^n=F_p(t)-1=tF_p^p(t),\label{4.3.49}
\end{align}
where $F_p(t)$ is the $p$th order Fuss-Catalan function satisfying 

\be\label{4.3.50}
F_p\left( t(1-t)^{p-1}\right)=\frac{1}{1-t}.
\ee
\end{theorem}

\begin{proof} 
For expression \eqref{4.3.48}, we find 

\begin{align*}
[t^n]\tilde g=& \tilde d_{n,0}=d_{pn+r,(p-1)n+r}=\binom{pn+r}{n}\\
=&[t^n](1+t)^{pn+r}=[t^n](1+t)^r((1+t)^p)^n\\
=&\left.[t^n]\frac{(1+w)^r}{1-t(d/dw)((1+w)^p)}\right|_{w=t(1+w)^p},
\end{align*}
which implies \eqref{4.3.48}. 

From \eqref{p-1-2} of Theorem \ref{thm:4.3.6-3} we know that 

\be\label{4.3.51}
(\tilde g, \tilde f)=\left( \frac{t\phi'(t)g(\phi)f(\phi)^r}{\phi^{r+1}}, f(\phi)\right),
\ee
where $\phi(t)=\overline{\frac{t^{p}}{(f(t))^{p-1}}}$, and $\bar h(t)$ is the compositional inverse of $h(t)$ $(h(0)=0$ and $h'(0)\not=0)$. Moreover, the generating function of the A-sequence of the new array $(\tilde g, \tilde f)$ is $(A(t))^p$ , where $A(t)$ is the generating function of the A-sequence of the given Riordan array $(g,f)$. By using the Lagrange Inverse Formula 
\[
[t^{n}](f(t))^k=\frac{k}{n}[t^{n-k}](A(t))^n,
\]
we have 

\[
[t^n]\tilde f=\frac{1}{n}[t^{n-1}](A(t))^{pn}=\frac{1}{n}[t^{n-1}](1+t)^{pn}=\frac{1}{n}\binom{pn}{n-1}.
\]
Therefore, 

\[
\tilde f=\sum^\infty_{n=1}\frac{(pn)!}{((p-1)n+1)!n!}t^n=\sum^\infty_{n=1} \frac{1}{pn+1}\binom{pn+1}{n}t^n=F_p(t)-1.
\]
Since the key equation \eqref{4.3.6} of the Fuss-Catalan function $F_p$ shows $F_p=1+tF^p_p$, we obtain \eqref{4.3.49}. From \eqref{4.3.51}, 

\[
f(\phi)=\tilde f(t)=tF^p_p(t).
\]
Therefore, noting $f(t)=t/(1-t)$ we get 

\[
\frac{t}{1-t}=f(t)=\overline{\phi}F^p_p\left( \overline{\phi}\right)=\frac{t^p}{(f(t))^{p-1}}F^p_p\left( \frac{t^p}{(f(t))^{p-1}}\right)=t(1-t)^{p-1}F^p_p\left( t(1-t)^{p-1}\right),
\]
and \eqref{4.3.50} follows from the comparison of the leftmost side and the rightmost side of the above equation. 
\end{proof}

For example, if $p=2$ and $r\geq 0$, then 

\[
\tilde f=tF^2_2(t)=t(C(t))^2.
\]
Since $w=t(1+w)^2$ has a solution 

\[
w=\frac{1-2t-\sqrt{1-4t}}{2t}=C(t)-1,
\]
we have 

\[
\tilde g=\left.\frac{(1+w)^{r+1}}{1-w}\right|_{w=t(1+w)^2}=\frac{(C(t))^{r+1}}{2-C(t)}=\frac{(C(t))^r}{\sqrt{1-4t}}=B(t) (C(t))^r,
\]
where $B(t)$ is the generating function for the central binomial coefficients. Thus, $(\tilde d_{n,k})_{n,k\geq 0}=(d_{2n+r, n+r+k})_{n,k\geq 0}$ is the Riordan array 

\[
(\tilde g, \tilde f)=\left( B(t)C^r, t(C(t))^2\right).
\]

We need one more property of Riordan arrays, which generalizes a well-known property of the Pascal triangle and is shown in Brietzke \cite{Bri}.

\begin{theorem}\label{thm:4.3.14}
Let $(d_{n,k})_{n,k\geq 0}=(g,f)$ be a Riordan array. Then for any integers $k\geq s\geq 1$ we have 
\be\label{4.3.46}
d_{n,k}=\sum^n_{j=s}d_{n-j,k-s}[t^j](f(t))^s.
\ee

Particularly, for $s=1$, $d_{n,k}=\sum^n_{j=1}f_{j}d_{n-j,k-1}$, where $f_j=[t^j]f(t)$. 
\end{theorem}
\begin{proof}
The $(n,k)$ entry of the Riordan array $(g,f)$ can be written as 
\begin{align*}
d_{n,k}=& [t^n]g(t)(f(t))^k=[t^n]g(t)(f(t))^{k-s}((f(t))^s\\
=&\sum^n_{j=s}\left([t^{n-j}]g(t)(f(t))^{k-s}\right)\left([t^j] (f(t))^s\right)\\
=&\sum^n_{j=s}d_{n-j,k-s}[t^j]((f(t))^s.
\end{align*}
\end{proof}
\begin{example}\label{ex:4.3.6}
If $(g,f)=(1/(1-t), t/(1-t))$, then $f_j=[t^j](t/(1-t))=1$ for all $j\geq 1$. We have the well-known identity 
\be\label{4.3.47}
\sum^n_{j=1}\binom{n-j}{k-1}=\binom{n}{k}.
\ee
More generally, for the Pascal triangle $(g,f)=(1/(1-t), t/(1-t))$, we have 
\[
[t^j](f(t))^s=[t^j]\frac{t^s}{(1-t)^s}=[t^{j-s}](1-t)^{-s}=[t^{j-s}]\sum_{i\geq 0}\binom{s+i-1}{i}t^i
=\binom{j-1}{s-1}.
\]
Consequently, \eqref{4.3.46} becomes the Chu-Vandermonde identity
\[
\sum^n_{j=s}\binom{n-j}{k-s}\binom{j-1}{s-1}=\binom{n}{k},
\]
which contains \eqref{4.3.47} as a special case. 
\end{example}

\begin{example}\label{ex:4.3.7}
For fixed integers $p\geq 2$ and $r\geq 0$, starting with the Pascal triangle and using Theorem \ref{thm:4.3.6-2}, we obtain the Riordan array $(\tilde g, \tilde f)$ with its $(n,k)$ entry as 

\[
\tilde d_{n,k}=\binom{pn+r}{(p-1)n+r+k}=\binom{pn+r}{n-k}
\]
possesses the formal power series $\tilde f(t)=tF^p_p(t)$. Thus,

\begin{align*}
[t^j](\tilde f(t))^s=&[t^{j-s}]F_p^{ps}(t)=[t^{j-s}]\frac{ps}{pn+ps}\binom{pn+ps}{n}\\
=&\frac{ps}{p(j-s)+ps}\binom{p(j-s)+ps}{j-s}=\frac{s}{j}\binom{pj}{j-s}.
\end{align*}
From the expression \eqref{4.3.46} of Theorem \ref{thm:4.3.14} we obtain the identity

\be\label{4.3.52}
\sum^n_{j=s}\frac{s}{j}\binom{pj}{j-s}\binom{p(n-j)+r}{n-j-k+s}=\binom{pn+r}{n-k}. 
\ee
Particularly, if $s=1$, then \eqref{4.3.52} becomes 

\[
\sum^n_{j=1}\frac{1}{pj+1}\binom{pj+1}{j}\binom{p(n-j)+r}{n-j-k+1}=\binom{pn+r}{n-k}
\]
and, finally, adding to the both sides $\binom{pn+r}{n-k+1}$, we have 

\[
\sum^n_{j=0}\frac{1}{pj+1}\binom{pj+1}{j}\binom{p(n-j)+r}{n-j-k+1}=\binom{pn+r+1}{n-k+1}.
\]

Setting $j=i+s$, $x=ps$, $y=pk-ps+r$, and replacing $n$ by $n+k$, identity \eqref{4.3.52} becomes formula (5.62) of \cite{GKP}:

\[
\sum^n_{i=0}\frac{x}{x+pi}\binom{x+pi}{i}\binom{y+p(n-i)}{n-i}=\binom{x+y+pn}{n}.
\]
Substituting $p=-q$, $x=r$, and $y+pn=p$, the above identity is equivalent to the {\it Gould identity}:
\[
\sum^n_{i=0}\frac{r}{r-qi}\binom{r-qi}{i}\binom{p+qi}{n-i}=\binom{r+p}{n}.
\]
\end{example}

\noindent{\bf Acknowledgements} 
\medbreak
The author wish to express his gratitude and appreciation to the referee and the editor for their helpful comments and remarks. 

\bigbreak


\begin{thebibliography}{99}
\bigbreak
\bibitem{Aig} M. S. Aigner, \textit{A Course in Enumeration} Graduate Texts in Mathematics, 238. Springer, Berlin, 2007.

\bibitem{Barry13} P. Barry, On the central coefficients of Riordan matrices, {\it J. Integer Seq.} 16 (2013) 13.5.1.

\bibitem{Barry} P. Barry, \textit{Riordan Arrays: A Primer}, LOGIC Press, Kilcock, 2016.

\bibitem{Barry19} P. Barry, On the halves of a Riordan array and their antecedents, {\it Linear Algebra Appl.} 582 (2019), 114--137. 

\bibitem{Bar} P. Barry, On the r-shifted central triangles of a Riordan array, https://arxiv.org/abs/1906.01328.

\bibitem{Bri} E. H. M. Brietzke, An identity of Andrews and a new method for the Riordan array proof of combinatorial identities, {\it Discrete Math.} 308 (2008), 4246--4262.

\bibitem{Com74} L. Comtet, {\it Advanced Combinatorics}, Dordrecht: Reidel, 1974.

\bibitem{GKP} R. L. Graham, D. E. Knuth, and O. Partashnik, {\it Concrete Mathematics},  Addison-Wesley Publishing Company, Reading, MA, 1994.

\bibitem{He13} T.-X. He, Parametric Catalan numbers and Catalan triangles, {\it Linear Algebra Appl.} 438 (2013), no. 3, 1467--1484.

\bibitem{He20} T.-X. He, $A$-sequence, $Z$-sequence, and $B$-sequences of Riordan matrices, \emph{Discrete Math.}, 343 (2020), no. 3, 111718, 18 pp. 

\bibitem{He20-2} T.-X. He, Half Riordan array sequences,{\it  Linear Algebra Appl.} 604 (2020), 236--264.

\bibitem{HS17} T.-X. He and L. W. Shapiro, Fuss-Catalan matrices, their weighted sums, and stabilizer subgroups of the Riordan group, \textit{Linear Algebra Appl.} 532 (2017), 25--42.

\bibitem{HS20} T.-X. He and L. W. Shapiro, Palindromes and pseudo-involution multiplication, \emph{Linear Algebra Appl.}, 593 (2020), 1--17. 

\bibitem{He18} T.-X. He, He, Sequence characterizations of double Riordan arrays and their compressions, {\it Linear Algebra Appl.} 549 (2018), 176--202.

\bibitem{HS} T. -X. He and R. Sprugnoli. Sequence characterization of Riordan Arrays, \emph{Discrete Math.}, 309 (2009), no. 12, 3962--3974.

\bibitem{Hsu15} L. C. Hsu, On a pair of operator series expansions implying a variety of summation formulas, {\it  Anal. Theory Appl.} 31 (2015), no. 3, 260--282. 

\bibitem{MRSV} D. Merlini, D. G. Rogers, R. Sprugnoli, and M. C. Verri, On some alternative characterizations of Riordan arrays, \textit{Canad. J. Math.} 49 (1997), no. 2, 301--320.

\bibitem{MSV}
D. Merlini, R. Sprugnoli, and M. C. Verri, Lagrange inversion: When and how, {\it Acta Appl. Math.}  94 (2006), 233--249. 

\bibitem{Rom}
S.M. Roman, {\it The Umbral Calculus}, New York, Acad. Press, 1984.

\bibitem{RomRot78} S. Roman and G.-C. Rota, {\it The Umbral Calculus}, Adv. in Math., 1978, 95--188.

\bibitem{Shapiro} L. W. Shapiro, Bijections and the Riordan group, Random generation of combinatorial objects and bijective combinatorics, {\it Theoret. Comput. Sci.} 307 (2003), no. 2, 403--413.

\bibitem{SGWW} L. W. Shapiro, S. Getu, W. J. Woan and L. Woodson, The Riordan group, \emph{Discrete Appl. Math.} 34 (1991), 229--239.

\bibitem{OEIS}
N. J. A. Sloane, The On-Line Encyclopedia of Integer Sequences, https://oeis.org/, founded in 1964. 

\bibitem{Spr94} R. Sprugnoli. Riordan arrays and combinatorial sums,
{\it Discrete Math.} 132 (1--3) (1994),  267--290.

\bibitem{Spr95} R. Sprugnoli. Riordan arrays and the Abel-Gould identity, 
{\it Discrete Math.} 142 (1--3) (1995), 213--233.

\bibitem{Stanley} R. P. Stanley, \textit{Enumerative Combinatorics}, Vol.~2, Cambridge University Press, 1999.

\bibitem{Sta} R. P. Stanley, \textit{Catalan Numbers}, Cambridge University Press, New York, 2015.

\bibitem{YXH}
S.-L. Yang, Y.-X. Xu, and T.-X. He, $(m,r)$-central Riordan arrays and their applications, {\it Czechoslovak Math. J.}  67(142) (2017), no. 4, 919--936.

\bibitem{YZYH}
S.-L. Yang, S.-N. Zheng, S.-P. Yuan, and T.-X. He, Schr\"oder matrix as inverse of Delannoy matrix, {\it Linear Algebra Appl.} 439 (2013), no. 11, 3605--3614. 

\bibitem{Zeleke} M. Zeleke, Riordan Arrays and their applications in
Combinatorics, parts 1 \& 2, \textit{YouTube},
https://www.youtube.com/watch?v=hdR24 ApU\_EM and
https://www.youtube.com/watch?v=coLIavPaW60.
\end{thebibliography}
\end{document}